\documentclass[12pt,a4paper,reqno]{amsart}
\usepackage[utf8]{inputenc}
\usepackage[english]{babel}
\usepackage{amsmath}
\usepackage{enumerate}
\usepackage{hyperref}
\usepackage{cite}
\usepackage{mathtools}
\usepackage{mathrsfs}
\usepackage{soul, xcolor}
\usepackage{wasysym}
\usepackage{amsmath,amssymb,xpatch,relsize,graphics,graphicx,amsthm}
\usepackage{array}  
\usepackage{lscape} 
\usepackage{booktabs} 
\usepackage{caption} 
\usepackage{float,graphicx}
\usepackage{subcaption}
\usepackage{epstopdf} 
\xapptocmd\normalsize{%
	\abovedisplayskip=12pt plus 3pt minus 9pt
	\abovedisplayshortskip=0pt plus 3pt
	\belowdisplayskip=12pt plus 3pt minus 9pt
	\belowdisplayshortskip=7pt plus 3pt minus 4pt
}{}{}
\theoremstyle{definition}
\newtheorem{definition}{Definition}[section]

\theoremstyle{plain}
\newtheorem{theorem}[definition]{Theorem}

\newtheorem{lemma}[definition]{Lemma}

\title{Various Geometric Properties of a class related to Special Functions}
\author[A. Kumar]{Ayush Kumar}
\address{Department of Mathematics, University of Delhi, Delhi--110 007, India}
\email{akumar7@maths.du.ac.in}
\author[N.K. Jain]{Naveen Kumar Jain}
\address {Department of Mathematics, Aryabhatta College, Delhi-110 021,India}
\email{naveenjain@aryabhattacollege.ac.in}
\numberwithin{equation}{section}
\allowdisplaybreaks



\begin{document}
  \maketitle

  \begin{abstract}
    In this paper, we will discuss several radii problems related to Wright Function involving four parameters.
  \end{abstract}

  \section{Introduction}
  Let \( f \) be an analytic function in the open unit disk \( \mathbb{D} = \mathbb{D}_1\) where \( \mathbb{D}_r=\{ z \in \mathbb{C} : |z| < r \} \), normalized by \( f(0) = 0 \) and \( f'(0) = 1 \). The class of these functions is denoted by $\mathcal{A}$. Two analytic functions $f$ and $g$ in $\mathbb{D}$ are related by subordination, denoted $f(z) \prec g(z)$, if there exists an analytic function $w$ in $\mathbb{D}$ with $w(0) = 0$ and $|w(z)| < 1$ for all $z \in \mathbb{D}$ such that $f(z) = g(w(z))$. Whenever $f(z) \prec g(z)$, it holds that $f(0) = g(0)$ and $f(\mathbb{D}) \subseteq g(\mathbb{D})$. Moreover, if $g$ is univalent in $\mathbb{D}$, then $f(z) \prec g(z)$ if and only if $f(0) = g(0)$ and $f(\mathbb{D}) \subseteq g(\mathbb{D})$.
  The $\mathcal{S}^* (\alpha)$- radius of a normalized function $f$ is given by
  \[
   \mathcal{S}^* (\alpha)-radius=\sup\left\{ r \in \mathbb{R}^+ : \Re\left(\dfrac{zf'(z)}{f(z)}\right)>\alpha, \alpha \in \mathbb{D}_r \right\}.
  \]
  Similarly, the $\mathcal{C}(\alpha)$- radius can be defined. Note that the classes $\mathcal{S}^*(\alpha)$  and $\mathcal{C}(\alpha)$ are classes of starlike and convex functions of order $\alpha$. From the subordination point of view, these radii can be generalized to the Ma and Minda classes defined as
  \[
    \mathcal{S}^*(\varphi) = \left\{ f \in \mathcal{A} : \frac{z f'(z)}{f(z)} \prec \varphi(z) \right\}
    \quad \text{and} \quad
    \mathcal{C}(\varphi) = \left\{ f \in \mathcal{A} : 1 + \frac{z f''(z)}{f'(z)} \prec \varphi(z) \right\},
  \]
  \text{where}
  \begin{itemize}
    \item $\varphi$ is analytic and univalent with $\Re\varphi(z) > 0$, $\varphi'(0) > 0$;
    \item $\varphi(\mathbb{D})$ is starlike with respect to $\varphi(0) = 1$ and symmetric about real axis.
  \end{itemize}

 \begin{definition}
Let \( f \in \mathcal{A} \) be a special function. Then \( S^*(\varphi) \)-radius and \( \mathcal{C}(\varphi) \)-radius of \( f \) are defined as follows:
\[
r^*_{\varphi}(f) = \sup \left\{ r \in \mathbb{R}^+ : \frac{z f'(z)}{f(z)} \in \varphi(\mathbb{D}),\ z \in \mathbb{D}_r \right\}
\]
and
\[
r^c_{\varphi}(f) = \sup \left\{ r \in \mathbb{R}^+ : 1 + \frac{z f''(z)}{f'(z)} \in \varphi(\mathbb{D}),\ z \in \mathbb{D}_r \right\},
\]
respectively.
\end{definition}

\begin{definition}
Suppose that $f$ is a member of the class $\mathcal{A}$. Assume $\gamma$ belongs to the interval $\left( -\frac{\pi}{2}, \frac{\pi}{2} \right)$ and $0 \leq \alpha < 1$.

\vspace{1em}
A function $f$ analytic in the unit disk is called \textit{$\gamma$-spirallike of order $\alpha$} if and only if
\[
\mathrm{Re} \left( e^{-i\gamma}\, \frac{z f'(z)}{f(z)} \right) > \alpha \cos\gamma
\]
holds throughout $\mathbb{D}$. The collection of all such functions shall be denoted by $S_p^\gamma(\alpha)$.

\vspace{2em}
A function is said to be a member of the family $CS_p^\gamma(\alpha)$ of \textit{convex $\gamma$-spirallike functions of order $\alpha$} provided that
\[
\mathrm{Re} \left( e^{-i\gamma} \left[ 1 + \frac{z f''(z)}{f'(z)} \right] \right) > \alpha \cos\gamma
\]
holds for all $z \in \mathbb{D}$.
\end{definition}
\vspace{2em}
\begin{definition}
Let $g$ represent the normalized version of a particular special function. Define the radius of $\gamma$-spirallikeness of order $\alpha$ for $g$ as
\[
R_{sp}(\gamma, \alpha)
= \sup \left\{ r \in \mathbb{R}^+ : \mathrm{Re} \left( e^{-i\gamma} \frac{z g'(z)}{g(z)} \right) > \alpha \cos\gamma,\; z \in \mathbb{D}_r \right\}.
\]
Similarly, the radius for convex $\gamma$-spirallike functions of order $\alpha$ is given by
\[
R_{sp}^c(\gamma, \alpha)
= \sup \left\{ r \in \mathbb{R}^+ : \mathrm{Re} \left( e^{-i\gamma} \left[ 1 + \frac{z g''(z)}{g'(z)} \right] \right) > \alpha \cos\gamma,\; z \in \mathbb{D}_r \right\}.
\]
\end{definition}
\begin{definition}
  The \textit{exponential radius of starlikeness} $\mathcal{R}_e^*(f)$ of a function $f$ defined on $\mathbb{D}_r$ is the largest positive real number $r$ that satisfies $|\log(zf'(z)/f(z))| < 1$ for all $z \in \mathbb{D}_r$. Similarly, the \textit{exponential radius of convexity} $\mathcal{R}_e^c(f)$ of a function $f$ is the largest positive real number $r$ that satisfies $|\log(1 + zf''(z)/f'(z))| < 1$ for all $z \in \mathbb{D}_r$.
\end{definition}
In 2023 \cite{kazim23}, Kazimoglu and Gangania studied the above radii problems for different classes of special functions. In 2024 \cite{kazim24}, they studied geometric properties of functions containing derivatives of Bessel function. Several other classes being studied by authors (see \cite{ravi2017}, \cite{naz2024}).

The following lemma (ref. art. \cite{deniz2017}) will be used in our main results.
\begin{lemma}\label{lem1}
Let $a>b>r\ge |z|$ and $\lambda\in[0,1]$. Then
\[
\left|\frac{z}{\,b-z\,}-\lambda\,\frac{z}{\,a-z\,}\right|
\;\le\;
\frac{r}{\,b-r\,}-\lambda\,\frac{r}{\,a-r\,}.
\]
As immediate consequences, one has
\[
\Re\!\left(\frac{z}{\,b-z\,}-\lambda\,\frac{z}{\,a-z\,}\right)
\;\le\;
\frac{r}{\,b-r\,}-\lambda\,\frac{r}{\,a-r\,},
\qquad
\Re\!\left(\frac{z}{\,b-z\,}\right)
\;\le\;
\left|\frac{z}{\,b-z\,}\right|
\;\le\;
\frac{r}{\,b-r\,}.
\]

If, in addition, $b>a>r\ge |z|$, then
\[
\left|\frac{1}{(a+z)(b-z)}\right|
\;\le\;
\frac{1}{(a-r)(b+r)}.
\]
\end{lemma}

\begin{lemma}\label{lem2}
For $\frac{1}{2} (1 + \frac{1}{e}) \leq a \leq \frac{1}{2}(1 + e)$, let $r_a$ be given by
\[
r_a :=
\begin{cases}
a - \frac{1}{e}, & \frac{1}{2} (1 + \frac{1}{e}) \leq a \leq \frac{1}{2} (e + \frac{1}{e}); \\
e - a, & \frac{1}{2} (e + \frac{1}{e}) \leq a \leq \frac{1}{2} (1 + e).
\end{cases}
\]
Then $1 \in \{w : |w - a| < r_a\} \subseteq \Omega_e$.
\end{lemma}

The inequality [\cite{ravi1997}, Lemma 3.2, p. 310]:
\begin{equation}\label{eql1}
\left| \frac{z}{z - z_k} + \frac{r^2}{R^2 - r^2} \right| \leq \frac{Rr}{R^2 - r^2}, \tag{1.1}
\end{equation}
where $|z| \leq r < 1$ and $|z_k| = R > r$ plays a vital role in proving our results. The following result will be also needed in our investigation which is a direct consequence of \eqref{eql1}.

\begin{lemma}\label{lem3}
For $|z| \leq r < 1$, $|z_k| = \alpha > r$ and $\beta > \alpha$, we have
\[
\left| \frac{z}{z - z_k} + \frac{r^2}{\alpha^2 - r^2} + \frac{r^2}{\beta^2 - r^2} \right| \leq \frac{\alpha r}{\alpha^2 - r^2} + \frac{\beta r}{\beta^2 - r^2}.
\]
\end{lemma}

\section{ \( S^*(\varphi) \)-radius and \( \mathcal{C}(\varphi) \)-radius of 4-Parameter Wright Function}

The general Wright function involving four parameters \{\cite{luchko98}, \cite{mehrez18}\} is defined as
\begin{equation}
\mathcal{W}_{(\mu,a),(\nu,b)}(z) = \sum_{k=0}^\infty \frac{z^k}{\Gamma(a + k\mu)\Gamma(b + k\nu)},
\qquad a, b \in \mathbb{C}, \quad \mu, \nu \in \mathbb{R}.
\end{equation}

For particular parameter choices such as $\mu, \nu > 0$, this series has been explored in detail by Wright. Various properties, including cases such as $b = \nu = 1$ and $-1 < \mu < 0$, have been investigated. When $\mu + \nu > 0$, it can be shown that the power series converges absolutely for any $z \in \mathbb{C}$. It is also established for $a, b \in \mathbb{C}$ and $0 < -\mu < \nu$ that $\mathcal{W}_{(\mu,a),(\nu,b)}(z)$ defines an entire function.

\vspace{1em}


Let $\phi_{(\mu,a),(\nu,b),n}$ be the n-th positive zero of derivative of function \[ \Psi_{(\mu,a),(\nu,b)}(z)=z^ {ab} \mathcal{W}_{(\mu,a),(\nu,b)}(-z^2), \] and let $\psi _{(\mu,a),(\nu,b),n}$ be the n-th zero of $\mathcal{W}_{(\mu,a),(\nu,b)}(-z^2)$ then the result (\cite{baricz18}, lemma 1) shows that for $\nu,\mu,a,b >0$ it can be written as
\begin{equation}\label{eq1}
    \Gamma(a)\Gamma(b) \mathcal{W}_{(\mu,a),(\nu,b)}(-z^2)= \prod_{n\geq 1}\left(1-\dfrac{z^2}{\psi_{(\mu,a),(\nu,b),n}^2} \right),
\end{equation}
  and all the respective zeroes will satisfy the property as
  \[
    \phi_{(\mu,a),(\nu,b),n} < \psi _{(\mu,a),(\nu,b),n} < \phi_{(\mu,a),(\nu,b),n+1} < \psi _{(\mu,a),(\nu,b),n+1}, \ \ (n \geq 1)
  \]
 Since, the function $\mathcal{W}_{(\mu,a),(\nu,b)}((-z^2))$ does not belong to class $\mathcal{A}$, so we require some normalization to it, let's define the functions
 \begin{equation}\label{eq2w}
\begin{cases}
    f_{(\mu,a),(\nu,b)}(z) = \left[ z^{ab} \Gamma(a)\Gamma(b) \mathcal{W}_{(\mu,a),(\nu,b)}(-z^2) \right]^{1/{ab}} \\[1.5ex]
    g_{(\mu,a),(\nu,b)}(z) = z \Gamma(a)\Gamma(b) \mathcal{W}_{(\mu,a),(\nu,b)}(-z^2) \\[1.5ex]
    h_{(\mu,a),(\nu,b)}(z) = z \Gamma(a)\Gamma(b) \mathcal{W}_{(\mu,a),(\nu,b)} (-z)
\end{cases}
\end{equation}
We will denote $\mathcal{W}_{(\mu,a),(\nu,b)}(-z^2)$ by $\mathfrak{W}_{(\mu,a),(\nu,b)}(z)$.
\begin{theorem}\label{th1w}
Let \( \mu, \nu, a, b \geq 0 \). The \( S^*(\varphi) \)-radius for the functions \( f_{(\mu,a),(\nu,b)} \), \( g_{(\mu,a),(\nu,b)} \), and \( h_{(\mu,a),(\nu,b)} \) given by \eqref{eq2w} are the least positive roots of the equations:
\begin{itemize}
    \item[(i)] \( r\, \mathfrak{W}'_{(\mu,a),(\nu,b)}(r) + ab \beta\mathfrak{W}_{(\mu,a),(\nu,b)}(r) = 0 \) with $ab \neq 0$.
    \item[(ii)] \( r\, {\mathfrak{W}'}_{(\mu,a),(\nu,b)}(r) + \beta\mathfrak{W}_{(\mu,a),(\nu,b)}(r) = 0 \)
    \item[(iii)] \( \sqrt{r}\, \mathfrak{W}'_{(\mu,a),(\nu,b)}(\sqrt{r}) + 2\beta\mathfrak{W}_{(\mu,a),(\nu,b)}(\sqrt{r}) = 0 \)
\end{itemize}
situated in  \( (0,\, \psi_{(\mu,a),(\nu,b), 1}) \), \( (0,\, \psi_{(\mu,a),(\nu,b), 1}) \), and  \( (0,\, \psi^2_{(\mu,a),(\nu,b), 1}) \) respectively,
where \( \varphi (-1) = 1- \beta \) and $\beta$ is the radius of the largest disk $\{ w: |w-1|<\beta \} \subseteq \varphi (\mathbb{D})$.
\end{theorem}

\begin{proof}
  By logarithmic differentiation of \eqref{eq2w} we get that
  \[
\left\{
\begin{aligned}
    &\frac{z f'_{(\mu,a),(\nu,b)}(z)}{f_{(\mu,a),(\nu,b)}(z)} = 1 + \frac{1}{ab} \frac{z \mathfrak{W}'_{(\mu,a),(\nu,b)}(z)}{\mathfrak{W}_{(\mu,a),(\nu,b)}(z)} = 1 - \frac{1}{ab} \sum_{n \geq 1} \frac{2z^2}{\psi^2_{(\mu,a),(\nu,b), n} - z^2} \\[2ex]
    &\frac{z g_{(\mu,a),(\nu,b)}'(z)}{g_{(\mu,a),(\nu,b)}(z)} = 1 + \frac{z \mathfrak{W}'_{(\mu,a),(\nu,b)}(z)}{\mathfrak{W}_{(\mu,a),(\nu,b)}(z)} = 1 - \sum_{n \geq 1} \frac{2z^2}{\psi^2_{(\mu,a),(\nu,b), n} - z^2} \\[2ex]
    &\frac{z h'_{(\mu,a),(\nu,b)}(z)}{h_{(\mu,a),(\nu,b)}(z)} = 1 + \frac{1}{2} \frac{\sqrt{z} \mathfrak{W}'_{(\mu,a),(\nu,b)}(\sqrt{z})}{\mathfrak{W}_{(\mu,a),(\nu,b)}(\sqrt{z})} = 1 - \sum_{n \geq 1} \frac{z}{\psi^2_{(\mu,a),(\nu,b), n} - z}
\end{aligned}
\right.
\]
 Taking \[ \frac{z f'_{(\mu,a),(\nu,b)}(z)}{f_{(\mu,a),(\nu,b)}(z)} = 1 + \frac{1}{ab} \frac{z \mathfrak{W}'_{(\mu,a),(\nu,b)}(z)}{\mathfrak{W}_{(\mu,a),(\nu,b)}(z)} = 1 - \frac{1}{ab} \sum_{n \geq 1} \frac{2z^2}{\psi^2_{(\mu,a),(\nu,b), n} - z^2} \] with $ ab > 0 $. Consider the continuous function given by \( K^{f_{(\mu,a),(\nu,b)}}: (0,\psi_{(\mu,a),(\nu,b),1}) \rightarrow \mathbb{R} \) defined as
 \begin{equation}\label{eq3w}
   K^{f_{(\mu,a),(\nu,b)}} (r) = \frac{1}{ab} \sum_{n \geq 1} \frac{2r^2}{\psi^2_{(\mu,a),(\nu,b), n} - r^2} - \beta
 \end{equation}
 then \ \[ \dfrac{d}{dr}K^{f_{(\mu,a),(\nu,b)}} (r) = \frac{1}{ab} \sum_{n \geq 1} \frac{4r^2 \psi^2_{(\mu,a),(\nu,b), n}}{(\psi^2_{(\mu,a),(\nu,b), n} - r^2)^2} > 0, \  \text{ for all } ab>0 \] and for \( r<\psi_{(\mu,a),(\nu,b),1} \). \\
 Also, \( K^{f_{(\mu,a),(\nu,b)}} (0)= -\beta < 0 \) and \( \lim_{r\rightarrow \psi_{(\mu,a),(\nu,b),1}} K^{f_{(\mu,a),(\nu,b)}} (r)= \infty \). Thus there exist a unique positive root  say $\rho_{\varphi}(f_{(\mu,a),(\nu,b)})$ of $K^{f_{(\mu,a),(\nu,b)}} (r)=0$ in $(0,\psi_{(\mu,a),(\nu,b),1})$. \\
 Now, let $\{ w: |w-1|<\beta \} \subseteq \varphi (\mathbb{D})$ be such that $\varphi (-1)=1-\beta$ therefore, it is known that (lemma \ref{lem1}) if $z \in \mathbb{C}$ and $\delta \in \mathbb{R}$ such that $|z|<r<\delta$ then
 \[
 Re \left( \dfrac{z}{\delta - z}\right) \leq \left| \dfrac{z}{\delta - z}\right| \leq \dfrac{|z|}{\delta - |z|}
 \]
 so, the inequality
 \begin{equation}\label{eq4w}
 \left| \dfrac{z f^{'}_{(\mu,a),(\nu,b)}(z)}{f_{(\mu,a),(\nu,b)} (z)} -1 \right|
 = \left| \frac{1}{ab} \sum_{n \geq 1} \dfrac{2z^2}{\psi^2_{(\mu,a),(\nu,b),n}-z^2} \right|
 = \frac{1}{ab} \sum_{n \geq 1} \dfrac{2r^2}{\psi^2_{(\mu,a),(\nu,b),n}-r^2} \subseteq \beta
 \end{equation}
 implies that $f_{(\mu,a),(\nu,b)} \in S^{*}(\varphi)$ in $|z|<\rho_{\varphi}(f_{(\mu,a),(\nu,b)})$.\\
 For sharpness, take $|z| \in \mathbb{D}$ such that $z=r=-\rho_{\varphi}(f_{(\mu,a),(\nu,b)})$ then using \eqref{eq3w} and \eqref{eq4w}, we get that
 \[
 \left| \dfrac{z f^{'}_{(\mu,a),(\nu,b)}(z)}{f_{(\mu,a),(\nu,b)} (z)} -1 = \beta \right|
 \]
 which means that \( \dfrac{z f^{'}_{(\mu,a),(\nu,b)}(z)}{f_{(\mu,a),(\nu,b)} (z)} \notin \varphi(\mathbb{D}), \ \forall |z| \geq \rho_{\varphi}(f_{(\mu,a),(\nu,b)}) \), it proves the sharpness part.\\
 Same reasoning will work for other two functions, that can be easily computed.
\end{proof}

\begin{theorem}\label{th2w}
 Let \( \mu, \nu, a, b \geq 0 \). The \( C(\varphi)-radius\) of the functions $f_{(\mu,a),(\nu,b)}$, $g_{(\mu,a),(\nu,b)}$, and $h_{(\mu,a),(\nu,b)}$ is given by the positive root of the equations:

\begin{itemize}
\item[(i)]  $r f''_{(\mu,a),(\nu,b)}(r) + \beta ab f'_{(\mu,a),(\nu,b)}(r) = 0 \quad \text{where } ab >0 $
\item[(ii)]  $r g''_{(\mu,a),(\nu,b)}(r) + \beta g'_{(\mu,a),(\nu,b)}(r) = 0 \quad \text{where } ab \geq 0$
\item[(iii)]  $r h''_{(\mu,a),(\nu,b)}(\sqrt{r}) + \beta h'_{(\mu,a),(\nu,b)}(\sqrt{r}) = 0 \quad \text{where } ab \geq 0$
\end{itemize}

situated in $(0, \psi_{(\mu,a),(\nu,b),1}), (0, \psi_{(\mu,a),(\nu,b),1}) \text{ and } (0, \psi^{2}_{(\mu,a),(\nu,b),1})$ respectively,
where $\varphi(-1)=1-\beta$ and $\beta$ is the radius of the largest disk $\{w:|w-1|<\beta \} \subseteq \varphi(\mathbb{D})$.
\end{theorem}

\begin{proof}
  From the functions defined in \eqref{eq2w} we have
\[
\frac{z f''_{(\mu,a),(\nu,b)}(z)}{f'_{(\mu,a),(\nu,b)}(z)} = 1 + \frac{z \mathcal{W}''_{(\mu,a),(\nu,b)}(z)}{\mathcal{W}'_{(\mu,a),(\nu,b)}(z)}+ (\frac{1}{ab}-1) \frac{z \mathcal{W}'_{(\mu,a),(\nu,b)}(z)}{\mathcal{W}_{(\mu,a),(\nu,b)}(z)} \]
\[
= 1 - \sum_{n \geq 1} \frac{2z^2}{\tilde{\psi}^2_{(\mu,a),(\nu,b), n} - z^2} - (\frac{1}{ab}-1) \sum_{n \geq 1} \frac{2z^2}{\psi^2_{(\mu,a),(\nu,b), n} - z^2}
\]
Using lemma \ref{lem1} we get that
\[
Re \left( 1+ \frac{z f''_{(\mu,a),(\nu,b)}(z)}{f'_{(\mu,a),(\nu,b)}(z)} \right) \leq \left| 1+ \frac{z f''_{(\mu,a),(\nu,b)}(z)}{f'_{(\mu,a),(\nu,b)}(z)} \right|
\]
implies that
\[
Re \left( \frac{z f''_{(\mu,a),(\nu,b)}(z)}{f'_{(\mu,a),(\nu,b)}(z)} \right) \leq \left| \frac{z f''_{(\mu,a),(\nu,b)}(z)}{f'_{(\mu,a),(\nu,b)}(z)} \right|
 \leq \]
  \[
  \sum_{n \geq 1} \frac{2z^2}{\tilde{\psi}^2_{(\mu,a),(\nu,b), n} - z^2} + (\frac{1}{ab}-1) \sum_{n \geq 1} \frac{2z^2}{\psi^2_{(\mu,a),(\nu,b), n} - z^2} = \frac{- r f''_{(\mu,a),(\nu,b)}(r)}{f'_{(\mu,a),(\nu,b)}(r)}
\]
holds in $|z| = r < \tilde{\psi}_{(\mu,a),(\nu,b), 1}$ for $ab \leq 1$.\\
Observe the inequality \[ \left| \frac{z f''_{(\mu,a),(\nu,b)}(z)}{f'_{(\mu,a),(\nu,b)}(z)} \right| \leq \frac{-r f''_{(\mu,a),(\nu,b)}(r)}{f'_{(\mu,a),(\nu,b)}(r)} \] also holds for $ab > 1$ in $r < \tilde{\psi}_{(\mu,a),(\nu,b), 1}$. Consider $\rho^{c}_{\phi}(f_{(\mu,a),(\nu,b)})$ to be the smallest positive root of the equation $$\frac{r f''_{(\mu,a),(\nu,b)}(r)}{f'_{(\mu,a),(\nu,b)}(r)}+\beta=0$$.\\
Now assume that \(\{ w: |w-1|<\beta \} \subseteq \varphi (\mathbb{D})\) be such that $\varphi(-1)=1-\beta$ then it follows that \[ \left| \frac{z f''_{(\mu,a),(\nu,b)}(z)}{f'_{(\mu,a),(\nu,b)}(z)} \right| \leq \frac{-r f''_{(\mu,a),(\nu,b)}(r)}{f'_{(\mu,a),(\nu,b)}(r)} \leq \beta \] holds for $|z|=r<\rho^{c}_{\phi}(f_{(\mu,a),(\nu,b)})$ for $ab > 0$ thus $f_{(\mu,a),(\nu,b)} \in C(\phi)$. \\
Now take $z=-\rho^{c}_{\phi}(f_{(\mu,a),(\nu,b)})$ then \[ \left| \frac{z f''_{(\mu,a),(\nu,b)}(z)}{f'_{(\mu,a),(\nu,b)}(z)} \right| = \frac{-r f''_{(\mu,a),(\nu,b)}(r)}{f'_{(\mu,a),(\nu,b)}(r)} = \beta \] implies that
$$1+ \frac{z f''_{(\mu,a),(\nu,b)}(z)}{f'_{(\mu,a),(\nu,b)}(z)} \notin \varphi (\mathbb{D}),\ \forall |z|=r \geq \rho^{c}_{\phi}(f_{(\mu,a),(\nu,b)})$$ which proves the sharpness part.
Other two parts can be proved similarly.
\end{proof}
\section{Exponential Radii of starlikeness and convexity of 4-parameter Wright Function}
\begin{theorem}
Let $\nu,\mu,a,b$ be positive real constants, and let $\psi_{(\mu,a), (\nu,b), 1}$ represent the first positive zero of the function $\mathfrak{W}_{(\mu,a), (\nu,b)}$.

\begin{enumerate}
    \item[(a)] For the function $f_{(\mu,a), (\nu,b)}$, the exponential starlikeness radius is given by $\mathcal{R}_e^*(f_{(\mu,a), (\nu,b)}) = t_{(\mu,a), (\nu,b), 1}$. Here, $t_{(\mu,a), (\nu,b), 1}$ is defined as the smallest positive value in the interval $(0, \psi_{(\mu,a), (\nu,b), 1})$ that satisfies:
    \[
    \frac{r \mathfrak{W}_{(\mu,a), (\nu,b)}'(r)}{\mathfrak{W}_{(\mu,a), (\nu,b)}(r)} + ab \left( 1 - \frac{1}{e} \right) = 0.
    \]

    \item[(b)] For the function $g_{(\mu,a), (\nu,b)}$, the exponential starlikeness radius is $\mathcal{R}_e^*(g_{(\mu,a), (\nu,b)}) = b1_{(\mu,a), (\nu,b), 1}$, where $b1_{(\mu,a), (\nu,b), 1} \in (0, \psi_{\rho, \beta, 1})$ corresponds to the first positive root of:
    \[
    \frac{r \mathfrak{W}_{(\mu,a), (\nu,b)}'(r)}{\mathfrak{W}_{(\mu,a), (\nu,b)}(r)} + 1 - \frac{1}{e} = 0.
    \]

    \item[(c)] The exponential starlikeness radius for $h_{(\mu,a), (\nu,b)}$ is denoted by $\mathcal{R}_e^*(h_{(\mu,a), (\nu,b)}) = c1_{(\mu,a), (\nu,b), 1}$. This value $c1_{(\mu,a), (\nu,b), 1}$ is the minimal positive root within $(0, \psi_{(\mu,a), (\nu,b), 1}^2)$ of the following equation:
    \[
    \frac{\sqrt{r} \mathfrak{W}_{(\mu,a), (\nu,b)}'(\sqrt{r})}{\mathfrak{W}_{(\mu,a), (\nu,b)}(\sqrt{r})} + 2 \left( 1 - \frac{1}{e} \right) = 0.
    \]
\end{enumerate}
\end{theorem}

\begin{proof}
For proving the results, we need to show that
\[
\left| \log \left( \frac{z f'_{(\mu,a), (\nu,b)}(z)}{f_{(\mu,a), (\nu,b)}(z)} \right) \right| < 1,
\qquad
\left| \log \left( \frac{z g'_{(\mu,a), (\nu,b)}(z)}{g_{(\mu,a), (\nu,b)}(z)} \right) \right| < 1
\quad \text{and} \quad
\left| \log \left( \frac{z h'_{(\mu,a), (\nu,b)}(z)}{h_{(\mu,a), (\nu,b)}(z)} \right) \right| < 1
\]
respectively for all $z$ in disks of some radius and these inequalities do not hold true
in any bigger disk other than that. Since, we know that
\[
\frac{z f'_{(\mu,a), (\nu,b)}(z)}{f_{(\mu,a), (\nu,b)}(z)}
=1+\frac{1}{ab}\frac{z\mathfrak{W}'_{(\mu,a), (\nu,b)}(z)}{\mathfrak{W}_{(\mu,a), (\nu,b)}(z)}
=1-\frac{2}{ab}\sum_{n=1}^{\infty}\frac{z^{2}}{\psi_{\rho,\beta,n}^{2}-z^{2}} .
\]

Using Lemma \ref{lem1}, we get the following inequality
\begin{equation}\label{eq3a}
\left|
\frac{z f'_{(\mu,a), (\nu,b)}(z)}{f_{(\mu,a), (\nu,b)}(z)}
-1+\frac{2}{ab}\sum_{n=1}^{\infty}\frac{r^{4}}{\psi_{(\mu,a), (\nu,b),n}^{4}-r^{4}}
\right|
\le
\frac{2}{ab}\sum_{n=1}^{\infty}\frac{\psi_{(\mu,a), (\nu,b),n}^{2}r^{2}}{\psi_{(\mu,a), (\nu,b),n}^{4}-r^{4}}
\tag{2.9}
\end{equation}
for all $|z|\le r<\psi_{(\mu,a), (\nu,b),1}$. Note that the equality holds in the above
equation when $z=ir$.
 Take $I_1:=(0,\psi_{(\mu,a), (\nu,b),1})$. Then the function $F_f:I_1\to\mathbb{R}$
defined by
\[
F_f(r)=1-\frac{2}{ab}\sum_{n=1}^{\infty}\frac{r^{4}}{\psi_{(\mu,a), (\nu,b), n}^{4}-r^{4}}
=\frac{1}{2}\left(
\frac{r f'_{(\mu,a), (\nu,b)}(r)}{f_{(\mu,a), (\nu,b)}(r)}
+
\frac{i r f'_{(\mu,a), (\nu,b)}(ir)}{f_{(\mu,a), (\nu,b)}(ir)}
\right).
\]
is continuous on $I_1$ also $F_f'(r)<0$ $\forall$ $r\in I_1$, thus $F_f$ is decreasing
on the interval $I_1$.\\
Clearly $F_f(0)=1>0$ and the function $F_f$ has the limit $-\infty$
when $r \nearrow \psi_{(\mu,a), (\nu,b),1}$. Implies that $F_f(a_{(\mu,a), (\nu,b),3})=0$
for some $t_{(\mu,a), (\nu,b),3}\in I_1$. Therefore by intermediate value theorem there
exists a unique real number $t_{(\mu,a), (\nu,b),2}$ in the interval $(0,t_{(\mu,a), (\nu,b),3})$
such that
\[
F_f(t_{(\mu,a), (\nu,b),2})=\frac{1}{2}\left(1+\frac{1}{e}\right).
\]
Note that if $0\le r\le t_{(\mu,a), (\nu,b),2}$, then
\[
\frac{1}{2}\left(1+\frac{1}{e}\right)\le F_f(r)\le 1.
\]
Now, define function $\xi_f$ as
\[
\xi_f(r):=\frac{1}{e}-\frac{r f'_{(\mu,a), (\nu,b)}(r)}{f_{(\mu,a), (\nu,b)}(r)}.
\]

Then the limit of the function $\xi_f$ is $\frac{1}{e}-1<0$ when $r\searrow0$
and $\infty$ when $r\nearrow\psi_{(\mu,a), (\nu,b),1}$.
Since $\psi_f'(r)>0$ $\forall$ $r\in I$ so $\xi_f$ is an increasing function of $r$. Therefore the function $\xi_f$ cuts the x-axis
exactly once in the interval $I_1$. By the computation
\[
\xi_f(t_{(\mu,a), (\nu,b),3})
=\frac{1}{e}-1+\frac{2}{ab}\sum_{n=1}^{\infty}
\frac{t_{(\mu,a), (\nu,b),3}^{2}}{\psi_{(\mu,a), (\nu,b),n}^{2}-t_{(\mu,a), (\nu,b),3}^{2}}
\]
\[
\ge
\frac{1}{e}-1+\frac{2}{ab}\sum_{n=1}^{\infty}
\frac{t_{(\mu,a), (\nu,b),3}^{4}}{\psi_{(\mu,a), (\nu,b),n}^{4}-t_{(\mu,a), (\nu,b),3}^{4}}
=
\frac{1}{e}F_f(t_{(\mu,a), (\nu,b),3})
=
\frac{1}{e}>0
\]
we sees that the function $\xi_f$ has a zero in the interval $(0,t_{(\mu,a), (\nu,b),3})$.
Let this zero be denoted by $t_{(\mu,a), (\nu,b),1}$. If $t_{(\mu,a), (\nu,b),1}\le t_{(\mu,a), (\nu,b),2}$,
then $\xi_f(r)\le0$ for all $r\le t_{(\mu,a), (\nu,b),1}$, that is
\[
\frac{2}{ab}\sum_{n=1}^{\infty}
\frac{\psi_{(\mu,a), (\nu,b),n}^{2}r^{2}}{\psi_{(\mu,a), (\nu,b),n}^{4}-r^{4}}
\le
F_f(r)-\frac{1}{e}
\]

for all $r\le t_{(\mu,a), (\nu,b),1}$. If $t_{(\mu,a), (\nu,b),2}<t_{(\mu,a), (\nu,b),1}<t_{(\mu,a), (\nu,b),3}$,
then
\[
0<F_f(r)<\frac{1}{2}\left(1+\frac{1}{e}\right)
\]
for $t_{(\mu,a), (\nu,b),2}<r<t_{(\mu,a), (\nu,b),1}$ and in this case, $1$ does not belong to the disk
\eqref{eq3a} which leads to a contradiction. Hence using Lemma 1, we see that the disk \eqref{eq3a}
lies inside the region $\Omega_e$ for all $r\le t_{(\mu,a), (\nu,b),1}$. This shows that
\[
R_e^*(f_{(\mu,a), (\nu,b)})\ge t_{(\mu,a), (\nu,b),1}.
\]
Also
\[
\left|
\log\left(
\frac{t_{(\mu,a), (\nu,b),1}\cdot f'_{(\mu,a), (\nu,b)}(t_{(\mu,a), (\nu,b),1})}
{f_{(\mu,a), (\nu,b)}(t_{(\mu,a), (\nu,b),1})}
\right)
\right|
=
\left|
\log\left(
\frac{1}{e}-\psi_f(t_{(\mu,a), (\nu,b),1})
\right)
\right|
=1
\]
proves that $R_e^*(f_{(\mu,a), (\nu,b)})=t_{(\mu,a), (\nu,b),1}$.\\

For part b, we know that
\[
\frac{z g'_{(\mu,a), (\nu,b)}(z)}{g_{(\mu,a), (\nu,b)}(z)}
=1+\frac{z\mathfrak{W}'_{(\mu,a), (\nu,b)}(z)}{\mathfrak{W}_{(\mu,a), (\nu,b)}(z)}
=1-\sum_{n=1}^{\infty}\frac{2z^{2}}{\psi_{(\mu,a), (\nu,b),n}^{2}-z^{2}}.
\]

Therefore from lemma \eqref{eql1},
\begin{equation}\label{eq3b}
\left|
\frac{z g'_{(\mu,a), (\nu,b)}(z)}{g_{(\mu,a), (\nu,b)}(z)}
-1+
\sum_{n=1}^{\infty}
\frac{2r^{4}}{\psi_{(\mu,a), (\nu,b),n}^{4}-r^{4}}
\right|
\le
\sum_{n=1}^{\infty}
\frac{2\psi_{(\mu,a), (\nu,b),n}^{2}r^{2}}
{\psi_{(\mu,a), (\nu,b),n}^{4}-r^{4}}
\tag{2.10}
\end{equation}

for all $|z|\le r<\psi_{(\mu,a), (\nu,b),1}$ with equality at $z=ir$. Taking $I_1:=(0,\psi_{(\mu,a), (\nu,b),1})$, it can be seen that the function
\[
F_g(r):=
1-\sum_{n=1}^{\infty}\frac{2r^{4}}{\psi_{(\mu,a), (\nu,b),n}^{4}-r^{4}}
=
\frac{1}{2}
\left(
\frac{r g'_{(\mu,a), (\nu,b)}(r)}{g_{(\mu,a), (\nu,b)}(r)}
+
\frac{i r g'_{(\mu,a), (\nu,b)}(ir)}{g_{(\mu,a), (\nu,b)}(ir)}
\right)
\]

is continuously decreasing on $I_1$. Following the similar way as part (a), we observe that $F_g(0)=1$ and $\lim_{r\nearrow\psi_{(\mu,a), (\nu,b),1}}F_g(r)=-\infty$.
This confirms the existence of a root of $C_g$ in the interval $I_1$, we denote it by $b1_{(\mu,a), (\nu,b),3}$, .
Again using intermediate value theorem we get a unique real number
$b1_{(\mu,a), (\nu,b),2}$ in the interval $(0,b1_{(\mu,a), (\nu,b),3})$ satisfying
\[
F_g(b1_{(\mu,a), (\nu,b),2})=\frac{1}{2}\left(1+\frac{1}{e}\right).
\]

Now for $0<r\le b1_{(\mu,a), (\nu,b),2}$, observe that
\[
\frac{1}{2}\left(1+\frac{1}{e}\right)\le F_g(r)\le 1.
\]

Consider a function defined by
\[
\xi_g(r):=\frac{1}{e}-\frac{r g'_{(\mu,a), (\nu,b)}(r)}{g_{(\mu,a), (\nu,b)}(r)}.
\]

Then $\lim_{r\searrow0}\xi_g(r)=\frac{1}{e}-1<0$ and
$\lim_{r\nearrow\psi_{(\mu,a), (\nu,b),1}}\xi_g(r)=+\infty$.
Also $\xi_g'(r)>0$ for all $r\in I_1$, this implies $\xi_g$ is an increasing
function of $r$. Therefore by intermediate value theorem, there exists a unique
root $b1_{(\mu,a), (\nu,b),1}$ of the function $\xi_g$ in the interval $I_1$, or more
precisely in the interval $(0,b1_{(\mu,a), (\nu,b),3})$ since

\[
\xi_g(b1_{(\mu,a), (\nu,b),3})
=
\frac{1}{e}-1+
\sum_{n=1}^{\infty}
\frac{2(b1)_{(\mu,a), (\nu,b),3}^{2}}
{\psi_{(\mu,a), (\nu,b),n}^{2}-{b1}^2_{(\mu,a), (\nu,b),3}}
\]

\[
\ge
\frac{1}{e}-1+
\sum_{n=1}^{\infty}
\frac{2(b1)_{(\mu,a), (\nu,b),3}^{4}}
{\psi_{(\mu,a), (\nu,b),n}^{4}-b1_{(\mu,a), (\nu,b),3}^{4}}
=
\frac{1}{e}-F_g(b1_{(\mu,a), (\nu,b),3})
=
\frac{1}{e}>0.
\]

This gives
\[
\sum_{n=1}^{\infty}
\frac{2\psi_{(\mu,a), (\nu,b),n}^{2}r^{2}}
{\psi_{(\mu,a), (\nu,b),n}^{4}-r^{4}}
\le
F_g(r)-\frac{1}{e}
\]

$\forall$ $r\le b1_{(\mu,a), (\nu,b),1}$. Observe that the zero $b1_{(\mu,a), (\nu,b),1}$ is always
smaller than $b1_{(\mu,a), (\nu,b),2}$, as if not, then $F_g(r)$ would
lie between $0$ and $\frac{1}{2}\left(1+\frac{1}{e}\right)$ which is not possible as
then the point $1$ would not belong to the disk \eqref{eq3b}. Hence Lemma 1 proves that
the disk \eqref{eq3b} is contained inside $\Omega_e$ for all $r\le b1_{(\mu,a), (\nu,b),1}$,
that is,
\[
R_e^*(g_{(\mu,a), (\nu,b)})\ge b1_{(\mu,a), (\nu,b),1}.
\]

Since
\[
\left|
\log
\left(
\frac{b1_{(\mu,a), (\nu,b),1}\cdot g'_{(\mu,a), (\nu,b)}(b1_{(\mu,a), (\nu,b),1})}
{g_{(\mu,a), (\nu,b)}(b1_{(\mu,a), (\nu,b),1})}
\right)
\right|
=1,
\]

therefore $R_e^*(g_{(\mu,a), (\nu,b)})=b1_{(\mu,a), (\nu,b),1}$.\\
 For part c, from the normalization and the decomposition, it is clear that
\[
\frac{z h'_{(\mu,a), (\nu,b)}(z)}{h_{(\mu,a), (\nu,b)}(z)}
=1+\frac{1}{2}\cdot \frac{\sqrt{z}\,\mathfrak{W}'_{(\mu,a), (\nu,b)}(\sqrt{z})}{\mathfrak{W}_{(\mu,a), (\nu,b)}(\sqrt{z})}
=1-\sum_{n=1}^{\infty}\frac{z}{\psi_{(\mu,a), (\nu,b),n}^{2}-z}.
\]

Using \eqref{eql1}, we have the inequality
\begin{equation}\label{eq3c}
\left|\frac{z h'_{(\mu,a), (\nu,b)}(z)}{h_{(\mu,a), (\nu,b)}(z)}-1+\sum_{n=1}^{\infty}\frac{r^{2}}{\psi_{(\mu,a), (\nu,b),n}^{4}-r^{2}}\right|
\le
\sum_{n=1}^{\infty}\frac{\psi_{(\mu,a), (\nu,b),n}^{2} r}{\psi_{(\mu,a), (\nu,b),n}^{4}-r^{2}}
\end{equation}

valid for all $|z|\le r<\psi_{(\mu,a), (\nu,b),1}^{2}$ with equality at $z=-r$ because of the minimum principle for the harmonic functions. Take $J:=(0,\psi_{(\mu,a), (\nu,b),1}^{2})$, then the function $L_h:J\to\mathbb{R}$ defined by
\[
L_h(r)=1-\sum_{n=1}^{\infty}\frac{r^{2}}{\psi_{(\mu,a), (\nu,b),n}^{4}-r^{2}}
=\frac{1}{2}\left(\frac{r h'_{(\mu,a), (\nu,b)}(r)}{h_{(\mu,a), (\nu,b)}(r)}-\frac{r h'_{(\mu,a), (\nu,b)}(-r)}{h_{(\mu,a), (\nu,b)}(-r)}\right)
\]

is continuously decreasing on the interval $J$. Note that $L_h(0)=1$ and the function $C_h$ takes the limit $-\infty$ when $r\uparrow\psi_{(\mu,a), (\nu,b),1}^{2}$ implying that $L_h(l_{(\mu,a), (\nu,b),3})=0$ for some $l_{(\mu,a), (\nu,b),3}\in J$. By intermediate value theorem, we have $L_h(l_{(\mu,a), (\nu,b),2})=\frac12\left(1+\frac1e\right)$ for some $l_{(\mu,a), (\nu,b),2}\in(0,l_{(\mu,a), (\nu,b),3})$. This gives
\[
\frac12\left(1+\frac1e\right)\le L_h(r)\le 1
\]
whenever $0\le r\le l_{(\mu,a), (\nu,b),2}$. define the function
\[
\xi_h(r):=\frac{1}{e}-\frac{r h'_{(\mu,a), (\nu,b)}(r)}{h_{(\mu,a), (\nu,b)}(r)},
\]

it behaves similar to the functions $\xi_f$ and $\xi_g$, that is, it has one root in the interval $J$, being an increasing function of $r$ and also observe that
\[
\lim_{r\searrow 0}\xi_h(r)=\frac{1}{e}-1<0
\quad\text{and}\quad
\lim_{r\nearrow\psi_{(\mu,a), (\nu,b),1}^{2}}\xi_h(r)=+\infty.
\]

Let $l_{(\mu,a), (\nu,b),1}$ be the unique zero of the function $\xi_h$. It can be easily seen that $l_{(\mu,a), (\nu,b),1}$ is always smaller than $l_{(\mu,a), (\nu,b),3}$ because
\[
\xi_h(l_{(\mu,a), (\nu,b),3})
=\frac{1}{e}-1+\sum_{n=1}^{\infty}\frac{l_{(\mu,a), (\nu,b),3}}{\psi_{(\mu,a), (\nu,b),n}^{2}-l_{(\mu,a), (\nu,b),3}}
\]

\[
\ge
\frac{1}{e}-1+\sum_{n=1}^{\infty}\frac{l_{(\mu,a), (\nu,b),3}^{2}}{\psi_{(\mu,a), (\nu,b),n}^{4}-l_{(\mu,a), (\nu,b),3}^{2}}
=
\frac{1}{e}-L_h(l_{(\mu,a), (\nu,b),3})
=
\frac{1}{e}>0.
\]

Therefore for all $r\le l_{(\mu,a), (\nu,b),1}$, we have
\[
\sum_{n=1}^{\infty}\frac{\psi_{(\mu,a), (\nu,b),n}^{2} r}{\psi_{(\mu,a), (\nu,b),n}^{4}-r^{2}}
\le
L_h(r)-\frac{1}{e}.
\]

Hence by applying Lemma 1*, the disk (3.1) lies inside the region $\Omega_e$ for $r\le l_{(\mu,a), (\nu,b),1}$. Following similar observation as the previous two parts the radius $l_{(\mu,a), (\nu,b),1}$ cannot be further improved so that
\[
R^{*}_{e}(h_{(\mu,a), (\nu,b)})=l_{(\mu,a), (\nu,b),1}.
\]
\end{proof}

\begin{theorem}
Suppose that $\mu, \nu > 0$.

\begin{itemize}
    \item[(a)] If $0 < a, b \leq 1$ and $\psi'_{(\mu, a), (\nu, b), 1}$ is the first positive zero of $\mathfrak{W}'_{(\mu, a), (\nu, b)}$, then $\mathcal{R}_e^c(f_{(\mu, a), (\nu, b)}) = s'_{(\mu, a), (\nu, b), 1}$, where $s'_{(\mu, a), (\nu, b), 1} \in (0, \psi'_{(\mu, a), (\nu, b), 1})$ is the smallest positive root of the equation
    \[
    1 + \frac{r \Psi''_{(\mu, a), (\nu, b)}(r)}{\Psi'_{(\mu, a), (\nu, b)}(r)} + \left( \frac{1}{ab} - 1 \right) \frac{r \Psi'_{(\mu, a), (\nu, b)}(r)}{\Psi_{(\mu, a), (\nu, b)}(r)} - \frac{1}{e} = 0,
    \]
    where $\Psi_{(\mu, a), (\nu, b)}(z) = z^{ab} \mathcal{W}_{(\mu, a), (\nu, b)}(-z^2)$.

    \item[(b)] If $ a, b > 0$ and $\eta'_{(\mu, a), (\nu, b), 1}$ is the first positive zero of $g'_{(\mu, a), (\nu, b)}$, then $\mathcal{R}_e^c(g_{(\mu, a), (\nu, b)}) = t'_{(\mu, a), (\nu, b), 1}$, where $t'_{(\mu, a), (\nu, b), 1} \in (0, \eta'_{(\mu, a), (\nu, b), 1})$ is the smallest positive root of the equation
    \[
    1 + \frac{r g''_{(\mu, a), (\nu, b)}(r)}{g'_{(\mu, a), (\nu, b)}(r)} - \frac{1}{e} = 0.
    \]

    \item[(c)] If $ a, b > 0$ and $\theta'_{(\mu, a), (\nu, b), 1}$ is the first positive zero of $h'_{(\mu, a), (\nu, b)}$, then $\mathcal{R}_e^c(h_{(\mu, a), (\nu, b)}) = u'_{(\mu, a), (\nu, b), 1}$, where $u'_{(\mu, a), (\nu, b), 1} \in (0, \theta'_{(\mu, a), (\nu, b), 1})$ is the smallest positive root of the equation
    \[
    1 + \frac{r h''_{(\mu, a), (\nu, b)}(r)}{h'_{(\mu, a), (\nu, b)}(r)} - \frac{1}{e} = 0.
    \]
\end{itemize}
\end{theorem}
\begin{proof}
To prove the theorem's statements, we need to show these inequalities
\[
\left| \log \left( 1 + \frac{zf''_{(\mu, a), (\nu, b)}(z)}{f'_{(\mu, a), (\nu, b)}(z)} \right) \right| < 1, \quad \left| \log \left( 1 + \frac{zg''_{(\mu, a), (\nu, b)}(z)}{g'_{(\mu, a), (\nu, b)}(z)} \right) \right| < 1
\]
and
\[
\left| \log \left( 1 + \frac{zh''_{(\mu, a), (\nu, b)}(z)}{h'_{(\mu, a), (\nu, b)}(z)} \right) \right| < 1
\]
holds respectively for all $z$ lying in disks of some particular radius and does not holds in any disk of larger radii.

 We will first consider part (a), the Weierstrass decomposition of the functions $\Psi_{(\mu, a), (\nu, b)}$ and $\Psi'_{(\mu, a), (\nu, b)}$:
\[
\Psi_{(\mu, a), (\nu, b)}(z) = \frac{z^{ab}}{\Gamma(a) \Gamma(b)} \prod_{n=1}^\infty \left( 1 - \frac{z^2}{\psi_{(\mu, a), (\nu, b), n}^2} \right) \quad \text{and} \quad \Psi'_{(\mu, a), (\nu, b)}(z) = \frac{z^{ab-1}}{\Gamma(a) \Gamma(b)} \prod_{n=1}^\infty \left( 1 - \frac{z^2}{\psi_{(\mu, a), (\nu, b), n}^{\prime 2}} \right)
\]
using logarithmic differentiation we get that
\begin{equation}\label{eq4a}
1 + \frac{zf''_{(\mu, a), (\nu, b)}(z)}{f'_{(\mu, a), (\nu, b)}(z)} = 1 - \sum_{n=1}^\infty \frac{2z^2}{\psi_{(\mu, a), (\nu, b), n}^{\prime 2} - z^2} - \left( \frac{1}{ab} - 1 \right) \sum_{n=1}^\infty \frac{2z^2}{\psi_{(\mu, a), (\nu, b), n}^2 - z^2},
\end{equation}
where $\psi_{(\mu, a), (\nu, b), n}$ and $\psi'_{(\mu, a), (\nu, b), n}$ denote the $n$th positive zero of the functions $\Psi_{(\mu, a), (\nu, b)}$ and $\Psi'_{(\mu, a), (\nu, b)}$ respectively, it is now clear that
\begin{equation}\label{eq4b}
\left| 1 + \frac{zf''_{(\mu, a), (\nu, b)}(z)}{f'_{(\mu, a), (\nu, b)}(z)} - C_f(r) \right| \leq \sum_{n=1}^\infty \frac{2\psi_{(\mu, a), (\nu, b), n}^{\prime 2} r^2}{\psi_{(\mu, a), (\nu, b), n}^{\prime 4} - r^4} + \left( \frac{1}{ab} - 1 \right) \sum_{n=1}^\infty \frac{2\psi_{(\mu, a), (\nu, b), n}^2 r^2}{\psi_{(\mu, a), (\nu, b), n}^4 - r^4}
\end{equation}
for all $|z| \leq r < \psi'_{(\mu, a), (\nu, b), 1} < \psi_{(\mu, a), (\nu, b), 1}$ and the equality holds if and only if $z = ir$, where
\begin{align*}
C_f(r) &:= 1 - \sum_{n=1}^\infty \frac{2r^4}{\psi_{(\mu, a), (\nu, b), n}^{\prime 4} - r^4} - \left( \frac{1}{ab} - 1 \right) \sum_{n=1}^\infty \frac{2r^4}{\psi_{(\mu, a), (\nu, b), n}^4 - r^4} \\
&= 1 + \frac{1}{2} \left( \frac{rf''_{(\mu, a), (\nu, b)}(r)}{f'_{(\mu, a), (\nu, b)}(r)} + \frac{irf''_{(\mu, a), (\nu, b)}(ir)}{f'_{(\mu, a), (\nu, b)}(ir)} \right).
\end{align*}
Since $0 < a,b \leq 1$, $1/ab > 1$ and therefore $C'_f(r) < 0$ for all $r \in (0, \psi'_{(\mu, a), (\nu, b), 1})$. This means the function $C_f$ is continuously decreasing on the interval $(0, \psi'_{(\mu, a), (\nu, b), 1})$ and it satisfies $C_f(0) = 1$. Also, $\lim_{r \nearrow \psi'_{(\mu, a), (\nu, b), 1}} C_f(r) = -\infty$ thus the function $C_f$ has a zero, say $s'_{(\mu, a), (\nu, b), 3}$, in the interval $(0, \psi'_{(\mu, a), (\nu, b), 1})$. Here also by using the intermediate value theorem gives the existence of a unique real number $s'_{(\mu, a), (\nu, b), 2} \in (0, s'_{(\mu, a), (\nu, b), 3})$ that satisfies $C_f(a'_{(\mu, a), (\nu, b), 2}) = \frac{1}{2} (1 + \frac{1}{e})$. If $r \in [0, s'_{(\mu, a), (\nu, b), 2}]$, then $C_f \in [\frac{1}{2} (1 + \frac{1}{e}), 1]$. Next, consider a function $\xi_f$ defined by
\[
\xi_f(r) := \frac{1}{e} - 1 - \frac{rf''_{(\mu, a), (\nu, b)}(r)}{f'_{(\mu, a), (\nu, b)}(r)}.
\]

Then $\lim_{r \searrow 0} \xi_f(r) = \frac{1}{e} - 1 < 0$ and $\lim_{r \nearrow \psi'_{(\mu, a), (\nu, b), 1}} \xi_f(r) = \infty$. Since $0 < a,b \leq 1$, we have $\xi'_f(r) > 0$ for all $r \in (0, \psi'_{(\mu, a), (\nu, b), 1})$ implying that the function $\xi_f$ is increasing in the interval $(0, \psi'_{(\mu, a), (\nu, b), 1})$. Consequently, the function $\xi_f$ has a unique zero $s'_{(\mu, a), (\nu, b), 1}$
in the interval $(0, \psi'_{(\mu, a), (\nu, b), 1})$. Moreover, the calculation
\begin{align*}
\xi_f(s'_{(\mu, a), (\nu, b), 3}) &= \frac{1}{e} - 1 + \sum_{n=1}^\infty \frac{2 s_{(\mu, a), (\nu, b), 3}^{\prime 2}}{\psi_{(\mu, a), (\nu, b), n}^{\prime 2} - s_{(\mu, a), (\nu, b), 1}^{\prime 2}} + \left( \frac{1}{\beta} - 1 \right) \sum_{n=1}^\infty \frac{2 s_{(\mu, a), (\nu, b), 3}^{\prime 2}}{\psi_{(\mu, a), (\nu, b), n}^2 - s_{(\mu, a), (\nu, b), 3}^{\prime 2}} \\
&\geq \frac{1}{e} - 1 + \sum_{n=1}^\infty \frac{2 s_{(\mu, a), (\nu, b), 3}^{\prime 4}}{\psi_{(\mu, a), (\nu, b), n}^{\prime 4} - s_{(\mu, a), (\nu, b), 3}^{\prime 4}} + \left( \frac{1}{\beta} - 1 \right) \sum_{n=1}^\infty \frac{2 s_{(\mu, a), (\nu, b), 3}^{\prime 4}}{\psi_{(\mu, a), (\nu, b), n}^4 - s_{(\mu, a), (\nu, b), 3}^{\prime 4}} \\
&= \frac{1}{e} - C_f(s'_{(\mu, a), (\nu, b), 3}) = \frac{1}{e} > 0
\end{align*}
shows that the function $\xi_f$ will always has its zero $s'_{(\mu, a), (\nu, b), 1}$ within the interval $(0, s'_{(\mu, a), (\nu, b), 3})$. Hence $\xi_f(r) \leq 0$ for all $r \in [0, s'_{(\mu, a), (\nu, b), 1}]$, or equivalently
\[
\sum_{n=1}^\infty \frac{2 \psi_{(\mu, a), (\nu, b), n}^{\prime 2} r^2}{\psi_{(\mu, a), (\nu, b), n}^{\prime 4} - r^4} - \left( \frac{1}{ab} - 1 \right) \sum_{n=1}^\infty \frac{2 \psi_{(\mu, a), (\nu, b), n}^2 r^2}{\psi_{(\mu, a), (\nu, b), n}^4 - r^4} \leq C_f(r) - \frac{1}{e}
\]
for all $r \leq s'_{(\mu, a), (\nu, b), 1}$. Observe that if $r \geq a'_{(\mu, a), (\nu, b), 2}$, then $0 \leq C_f(r) < \frac{1}{2} (1 + \frac{1}{e})$ and in this case, the point 1 does not belong to the disk \eqref{eq4b} denying the fact that it does. Thus Lemma \ref{lem2} proves that the disk \eqref{eq4b} lies inside $\Omega_e$ for all $r \leq s'_{(\mu, a), (\nu, b), 1} \leq s'_{(\mu, a), (\nu, b), 2}$, that is, $\mathcal{R}_e^c(f_{(\mu, a), (\nu, b)}) \geq s'_{(\mu, a), (\nu, b), 1}$. But since
\[
\left| \log \left( 1 + \frac{s'_{(\mu, a), (\nu, b), 1} \cdot f''_{(\mu, a), (\nu, b)}(s'_{(\mu, a), (\nu, b), 1})}{f'_{(\mu, a), (\nu, b)}(s'_{(\mu, a), (\nu, b), 1})} \right) \right| = 1,
\]
which follows that $\mathcal{R}_e^c(f_{(\mu, a), (\nu, b)}) = s'_{(\mu, a), (\nu, b), 1}$.\\
Other parts can be proved similar fashion.
\end{proof}
\section{  Radius of $\gamma$-spirallikeness of order $\alpha$ of four parameter Wright Function}

\begin{theorem}\label{th4w}
Let \( a,b,\mu,\nu > 0 \). The \( R_{sp}(\gamma,\alpha) \)-radius for the functions \( f_{(\mu,a),(\nu,b)} \), \( g_{(\mu,a),(\nu,b)} \), and \( h_{(\mu,a),(\nu,b)} \) given by \eqref{eq2w} are the least positive roots of the equations:
\begin{itemize}
    \item[(i)] \( r\, \mathfrak{W}_{(\mu,a),(\nu,b)}'(r) + ab (1-\alpha)cos \gamma \mathfrak{W}_{(\mu,a),(\nu,b)}(r) = 0 \)
    \item[(ii)] \( r\, {\mathfrak{W}}_{(\mu,a),(\nu,b)}'(r) + (1-\alpha)cos \gamma \mathfrak{W}_{(\mu,a),(\nu,b)}(r) = 0 \)
    \item[(iii)] \( \sqrt{r}\, \mathfrak{W}_{(\mu,a),(\nu,b)}'(\sqrt{r}) + 2(1-\alpha)cos \gamma \mathfrak{W}_{(\mu,a),(\nu,b)}(\sqrt{r}) = 0 \)
\end{itemize}
situated in  \( (0,\, \psi_{(\mu,a),(\nu,b), 1}) \), \( (0,\, \psi_{(\mu,a),(\nu,b), 1}) \), and  \( (0,\, \psi^2_{(\mu,a),(\nu,b), 1}) \) respectively,
where \( \varphi (-1) = 1- \beta \) and $\beta$ is the radius of the largest disk $\{ w: |w-1|<\beta \} \subseteq \varphi (\mathbb{D})$.\\
\end{theorem}

\begin{proof}
  Let $\alpha \in [0,1)$ and $\gamma \in \left(-\tfrac{\pi}{2},\tfrac{\pi}{2}\right)$, $\mu,\nu,a,b>0$ and let $\{\psi_{(\mu,a),(\nu,b),n}\}_{n\ge1}$ denotes the positive zero sequence of $\mathcal{W}_{(\mu,a),(\nu,b)}(-z^2)$. \\

We aim to establish the spiral-like inequalities
\[
\mathrm{Re}\!\left(e^{-i\gamma}\,\frac{z f'_{(\mu,a),(\nu,b)}(z)}{f_{(\mu,a),(\nu,b)}(z)}\right) > \alpha \cos\gamma,\qquad
\mathrm{Re}\!\left(e^{-i\gamma}\,\frac{z g'_{(\mu,a),(\nu,b)}(z)}{g_{(\mu,a),(\nu,b)}(z)}\right) > \alpha \cos\gamma,
\]
and
\[
\mathrm{Re}\!\left(e^{-i\gamma}\,\frac{z h'_{(\mu,a),(\nu,b)}(z)}{h_{(\mu,a),(\nu,b)}(z)}\right) > \alpha \cos\gamma,
\]
valid respectively in the disks $\mathbb{D}_{R_{sp}(\gamma,\alpha;f)}$, $\mathbb{D}_{R_{sp}(\gamma,\alpha;g)}$, and $\mathbb{D}_{R_{sp}(\gamma,\alpha;h)}$, with sharpness in the sense that none of the three holds in any larger disk.

By using  the classical inequality from lemma \ref{lem1} for every $|z|=r<\psi_{(\mu,a),(\nu,b),1}$,
\begin{equation*}
\mathrm{Re}\!\left(\frac{z^2}{\psi_{(\mu,a),(\nu,b),n}^{2}-z^{2}}\right)
\le \left| \frac{z^{2}}{\psi_{(\mu,a),(\nu,b),n}^{2}-z^{2}} \right|
\le \frac{|z|^{2}}{\psi_{(\mu,a),(\nu,b),n}^{2}-|z|^{2}}.
\end{equation*}
we obtain
\begin{align*}
\mathrm{Re}\!\left(e^{-i\gamma}\,\frac{z f'_{(\mu,a),(\nu,b)}(z)}{f_{(\mu,a),(\nu,b)}(z)}\right)
&= \mathrm{Re}(e^{-i\gamma}) - \frac{1}{ab}\,\mathrm{Re}\!\left( e^{-i\gamma} \sum_{n\ge1}\frac{2z^{2}}{\psi_{(\mu,a),(\nu,b),n}^{2}-z^{2}} \right) \\
&\ge \cos\gamma - \frac{1}{ab} \sum_{n\ge1}\frac{2r^{2}}{\psi_{(\mu,a),(\nu,b),n}^{2}-r^{2}}.
\end{align*}
Hence
\begin{equation*}
\mathrm{Re}\!\left(e^{-i\gamma}\,\frac{z f'_{(\mu,a),(\nu,b)}(z)}{f_{(\mu,a),(\nu,b)}(z)}\right)
\ge \cos\gamma - \frac{2r^{2}}{ab} \sum_{n\ge1}\frac{1}{\psi_{(\mu,a),(\nu,b),n}^{2}-r^{2}}.\\
=\dfrac{r f'_{(\mu,a),(\nu,b)}(r)}{f_{(\mu,a),(\nu,b)}(r)}+cos \gamma -1
\end{equation*}
The same reasoning for $g_{(\mu,a),(\nu,b)}, h_{(\mu,a),(\nu,b)}$ with lemma \ref{lem1} yields
\begin{equation*}
\mathrm{Re}\!\left(e^{-i\gamma}\,\frac{z g'_{(\mu,a),(\nu,b)}(z)}{g_{(\mu,a),(\nu,b)}(z)}\right)
\ge \cos\gamma - 2r^{2} \sum_{n\ge1}\frac{1}{\psi_{(\mu,a),(\nu,b),n}^{2}-r^{2}}.
\end{equation*}
we similarly obtain for $|z|=r<\psi_{(\mu,a),(\nu,b),1}$
\begin{equation*}
\mathrm{Re}\!\left(e^{-i\gamma}\,\frac{z h'_{(\mu,a),(\nu,b)}(z)}{h_{(\mu,a),(\nu,b)}(z)}\right)
\ge \cos\gamma - \sum_{n\ge1}\frac{r}{\psi_{(\mu,a),(\nu,b),n}^{2}-r}.
\end{equation*}

Consider for the $f$-case, the equality holds when $|z|=r$, Thus, for $r \in$  \( (0,\, \psi_{(\mu,a),(\nu,b), 1}) \) it follows that
\[
\inf_{|z|<r}\left\{\mathrm{Re}\!\left(e^{-i\gamma}\,\frac{z f'_{(\mu,a),(\nu,b)}(z)}{f_{(\mu,a),(\nu,b)}(z)}\right)\right\}\\
=\dfrac{r f'_{(\mu,a),(\nu,b)}(r)}{f_{(\mu,a),(\nu,b)}(r)}+(1-\alpha)cos \gamma -1
\]
Now, consider the auxiliary function on $\Theta_f:(0,\psi_{(\mu,a),(\nu,b),1})\rightarrow \mathbb{R}$ defined by
\[
\Theta_f(r) :=\dfrac{r f'_{(\mu,a),(\nu,b)}(r)}{f_{(\mu,a),(\nu,b)}(r)}+(1-\alpha)cos \gamma -1 \\
=(1-\alpha)cos \gamma -\frac{1}{ab}\sum_{n\ge1}\frac{2r^2}{\psi_{(\mu,a),(\nu,b),n}^{2}-r^2}
\]
is strictly decreasing as $\Theta'_f(r)<0$.
Since,
\[
\Theta_f'(r) = -\frac{1}{ab} \sum_{n\ge1} \frac{4r\,\psi_{(\mu,a),(\nu,b),n}^{2}}{(\psi_{(\mu,a),(\nu,b),n}^{2}-r^{2})^{2}} < 0,
\]
so $\Theta_f$ is strictly decreasing on $(0,\psi_{(\mu,a),(\nu,b),1})$. Moreover,
\[
\lim_{r\downarrow 0}\Theta_f(r)=(1-\alpha)\cos\gamma>0,\qquad
\lim_{r\uparrow \psi_{(\mu,a),(\nu,b),1}}\Theta_f(r)=-\infty.
\]
The minimum principle for harmonic functions implies that the required inequality for $f_{(\mu,a),(\nu,b)}$ holds iff $z \in $ $\mathbb{D}_{R_{sp}(\gamma,\alpha;f)}$, where $\mathbb{D}_{R_{sp}(\gamma,\alpha;f)}$ is the smallest positive root of $\dfrac{r f'_{(\mu,a),(\nu,b))}}{f_{(\mu,a),(\nu,b))}}=1-(1-\alpha)cos \gamma$.
Hence there exists a unique $R_{sp}(\gamma,\alpha;f)\in(0,\psi_{(\mu,a),(\nu,b),1})$ such that $\Theta_f(R_{sp}(\gamma,\alpha;f))=0$, and the inequality for $f$ holds precisely for $|z|<R_{sp}(\gamma,\alpha;f)$. The defining equation reads
\[
(1-\alpha)\cos\gamma=\frac{2{R_{sp}(\gamma,\alpha;f)}^{2}}{(ab)} \sum_{n\ge1}\frac{1}{\psi_{(\mu,a),(\nu,b),n}^{2}-{R_{sp}(\gamma,\alpha;f)}^{2}}.
\]
Reasoning along the same lines the other two parts follows.
\end{proof}

\section*{Declarations}

\textbf{Conflict of interest}  The author declare that he has no conflict of interest regarding the publication of this paper.

\textbf{Funding}
First author is supported by the UGC-Junior Research Fellowship since 2023.

\textbf{Data availability statement} Data sharing not applicable to this article as no datasets were generated or analysed during the current study.

\bibliographystyle{amsplain}

\end{document}